\documentclass{amsart}
\usepackage{calc}
\usepackage{ifthen}
\usepackage{epsfig}
\usepackage{latexsym}
\usepackage{amssymb,amsmath}
\newif\ifskip
\skiptrue
\newtheorem{theorem}{Theorem}[section]
\newtheorem{proposition}[theorem]{Proposition}
\newtheorem{lemma}[theorem]{Lemma}
\newtheorem{remark}{Remark}[section]
\ifskip
\else

\fi
\newenvironment{renumerate}{\begin{enumerate}}{\end{enumerate}}

\renewcommand{\bar}{\overline}

\newcommand{\MSOL}{\mathrm{MSOL}}

\newcommand{\N}{{\mathbb N}}

\newcommand{\R}{{\mathbb R}}

\begin{document}
\title[Tropical Graph Parameters]{
Finiteness Conditions for Graph Algebras  \\over Tropical Semirings$^*$} 
\thanks{$^*$ Accepted for presentation at FPSAC 2014 (Chicago, June 29 -July 3, 2014)}

\author[N. Labai]{Nadia  Labai$^1$}
\thanks{$^1$ Partially supported by a grant of the graduate school of the Technion.} 

\address{Computer Science Department, Technion--IIT, Haifa, Israel}
\email{labai@.cs.technion.ac.il}

\author[J.A. Makowsky]{Johann A. Makowsky$^2$}
\thanks{$^2$ Partially supported by a grant of Technion Research Authority.} 
\address{Computer Science Department, Technion--IIT, Haifa, Israel}
\email{janos@.cs.technion.ac.il}

\maketitle
\begin{center}
\today
\end{center}
\begin{abstract}
Connection matrices  for graph parameters
with values in a field have been introduced by
M. Freedman, L. Lov{\'a}sz and A. Schrijver (2007).
Graph parameters with connection matrices of finite rank  
can be computed in polynomial time on graph classes of bounded tree-width.
We introduce join matrices, a  generalization of connection matrices, 
and allow graph parameters to take values 
in the tropical rings (max-plus algebras)
over the real numbers. We show that rank-finiteness of join matrices implies that  these graph parameters
can be computed in polynomial time on graph classes of bounded clique-width.
In the case of graph parameters with values in arbitrary commutative semirings,
this remains true for graph classes of bounded linear clique-width.
B. Godlin, T. Kotek and J.A. Makowsky (2008) showed 
that definability of a graph parameter in Monadic Second Order Logic
implies rank finiteness.
We also show that there are uncountably many integer valued graph parameters with connection matrices or join matrices
of fixed finite rank. This shows that rank finiteness is a much weaker assumption than any definability assumption.
\end{abstract}

\newif\ifdense
\densefalse
\newcommand{\srg}{{\mathcal{S}\langle \mathcal{G} \rangle}}
\newcommand{\sr}{\mathcal{S}}
\newcommand{\cR}{\mathcal{R}}
\newcommand{\cI}{\mathcal{I}}
\newcommand{\cS}{\mathcal{S}}
\newcommand{\cB}{\mathcal{B}}
\newcommand{\cM}{\mathcal{M}}
\newcommand{\cQ}{\mathcal{Q}}
\newcommand{\cG}{\mathcal{G}}
\newcommand{\cCG}{\mathcal{CG}}
\newcommand{\cF}{\mathcal{F}}
\newcommand{\cK}{\mathcal{K}}
\newcommand{\Tr}{\mathcal{T}}
\newcommand{\cT}{\mathcal{T}}
\newcommand{\Ha}{\mathrm{H}}
\newcommand{\LEMS}{\mathrm{LEMS}}
\newcommand{\MSOLEVAL}{\mathrm{MSOLEVAL}}
\newcommand{\CMSOLEVAL}{\mathrm{CMSOLEVAL}}
\newcommand{\RMSOL}{\mathrm{RMSOL}}
\newcommand{\MH}{\mathrm{MH}}
\newcommand{\CW}{\mathrm{CW}}
\newcommand{\TW}{\mathrm{TW}}
\newcommand{\PW}{\mathrm{PW}}
\newcommand{\LCW}{\mathrm{LCW}}
\newcommand{\lb}{\left}
\newcommand{\rb}{\right}
\newif\iflata
\latafalse
\newif\ifskip
\skiptrue

\section{Introduction and Summary}

Connection matrices of graph parameters with values in a field $\cK$ have been introduced by
M. Freedman, L. Lov{\'a}sz and A. Schrijver (2007).
Graph parameters with connection matrices of finite rank exhibit many nice properties.
In particular, as was shown by L. Lov\'asz, \cite[Theorem 6.48]{bk:Lovasz-hom},
they can be computed in polynomial time on graph classes of bounded tree-width.
This is a logic-free version of the celebrated theorem by B. Courcelle, 
cf. \cite[Chapter 6.5]{bk:DowneyFellows99} and \cite[Chapter 11.4-5]{bk:FlumGrohe2006}.
The theorem is proved using the formalism of graph algebras as developed in \cite{bk:Lovasz-hom}.

In this paper we introduce join matrices, a  generalization of connection matrices, which will allow
us to replace the condition on tree-width to weaker conditions involving clique-width.
Courcelle's theorem was extended to this case in \cite{pr:CourcelleMakowskyRoticsWG98,ar:Oum2005}.
Furthermore we study graph parameters which take values 
in the tropical semirings $\cT_{max}$ and $\cT_{min}$ (max-plus algebras)
over the real numbers, as opposed to values in a field. 
We shall call them tropical graph parameters in contrast to real graph parameters.

There are several notions of rank for matrices over commutative semirings. All of them coincide in the case of a field,
and some of them coincide in the tropical case, \cite{bk:Butkovic2010,bk:Guterman2009,ar:Cuninghame-GreenButkovic2004}.
We shall work with two specific notions: row-rank in the tropical case, and a finiteness condition introduced by
G. Jacob \cite{phd:Jacob}, which we call J-finiteness, in the case of arbitrary commutative semirings.

A typical example of a tropical graph parameter with finite row-rank of its connection matrix
is $\omega(G)$, the maximal size of a clique in a graph $G$.
If viewed as a real graph parameter, its connection matrix has infinite rank.

\subsection*{Main results}
We adapt the formalism of graph algebras to tropical semirings with an inner product derived 
from the join matrices. Superficially this adaption may seem straightforward. However, there are several
complications to be overcome: 
\ifdense
(i) 
the definition of the join matrix, 
(ii) 
the choice of the finiteness condition on the join matrices, and 
(iii) 
the choice of the definition of the quotient algebra.
\else
\begin{renumerate}
\item 
the definition of the join matrix, 
\item 
the choice of the finiteness condition on the join matrices, and 
\item 
the choice of the definition of the quotient algebra.
\end{renumerate}
\fi 
Our main results are:
\begin{description}
\item[(Theorem \ref{th:cw-poly})] 
We show 
that row-rank finiteness of join matrices implies that  tropical graph parameters
can be computed in polynomial time on graph classes of bounded clique-width.
\item[(Theorem \ref{th:lcw-poly})]
A similar result holds 
in arbitrary commutative semirings 
when we replace row-rank finiteness with
J-finiteness and  bounded clique-width with bounded linear clique-width.
\end{description}

It was shown by B. Godlin, T. Kotek and J.A. Makowsky (2008) 
that definability of the graph parameter in Monadic Second Order Logic
implies rank finiteness.
\begin{description}
\item[(Theorems \ref{th:many-s},\ref{th:many-c})]
We show 
that there are uncountably many integer valued graph parameters with connection matrices or join matrices
of fixed finite rank. This shows that (row)-rank finiteness is a much weaker assumption than any definability assumption.
\end{description}

It is well known that graph classes of bounded tree-width are also of bounded clique-width, 
therefore we restrict our presentation to the case of bounded clique-width.
All results stated in this paper for tropical or arbitrary commutative semirings hold for fields as well.

\subsection*{Outline of the paper}
In Section \ref{se:prerequisites} we give the background on $k$-graphs, $k$-colored graphs, tree-width
and clique width. In Section \ref{se:Hankel} we introduce join-matrices, and more generally, Hankel matrices and
their ranks. In Section \ref{se:many} we show that there are uncountably many graph parameters with Hankel matrices
of fixed finite rank. In Section \ref{se:algebras} we construct the graph algebras for join-matrices of finite row-rank.
Finally, in Section \ref{se:CW} we show our main theorems for graph classes of bounded (linear) clique-width.
In Section \ref {se:conclu} we discuss our achievements and remaining open problems.

\section{Prerequisites}
\label{se:prerequisites}
\subsection{$k$-graphs and $k$-colored graphs}
Let $k \in \N$.
\\
A $k$-graph 
is a graph $G=(V(G), E(G))$ together with a partial map $\ell: [k] \rightarrow V(G)$.
\\
$\ell$ is called a {\em labeling} and the images of $\ell$ are called labels.

A $k$-colored graph
is a graph $G=(V(G), E(G))$ together with a map \\
$C: [k] \rightarrow 2^{V(G)}$.
$C$ is called a {\em coloring} and the images of $C$ are called colors.

$\ell$ and $C$ are often required to be injective, but this is not necessary.
If $\ell$ is partial not all labels in $[k]$ are assigned values in $V(G)$.
This corresponds to $C$ having as values the empty set in $V(G)$.
The labeling
$\ell$ can be viewed as a special case of the coloring $C$, 
where $C(i)$ is a singleton for all $i \in [k]$.

We denote the class of graphs by  $\cG$, the class of $k$-graphs by $\cG_k$,
and the class of $k$-colored graphs by $\cCG_k$.

\subsection{Gluing and joining}

We consider binary operations $\Box$ on $k$-graphs, resp. $k$-colored graphs. 
Specific examples are the following versions of gluing and joining,
but if not further specified, $\Box$ can be any isomorphism preserving 
binary operation.

Two
$k$-graphs 
$(G_1, \ell_1)$ and $(G_2, \ell_2)$ 
can be {\em glued together} producing a $k$-graph 
$(G, \ell) = (G_1, \ell_1) \sqcup_k (G_2, \ell_2)$
by taking the disjoint union of
$G_1$ and $G_2$ and $\ell_1$ and $\ell_2$ and identifying elements with the same label.

For two $k$-colored graphs
$(G_1, C_1)$ and $(G_2, C_2)$ we have similar operations.
Let $i,j \in [k]$ be given. We define their {\em $(i,j)$-join} by
$$
\bar{\eta}_{i,j}((G_1, C_1),(G_2, C_2)) 
= (G,C) 
$$
by taking  disjoint unions for
\ifdense

(i) $V(G)= V(G_1) \sqcup V(G_2)$ and

(ii)
for all $i \in [k]$
$C(i) =  C_1(i) \sqcup C_2(i)$, and

(iii)
$E(G)= E(G_1) \sqcup E(G_2) \cup \{ (u,v) \in V(G): u \in C(i), v \in C(j)\}$,
which connects in the disjoint union all vertices in $C(i)$ with all vertices in $C(j)$.

\else
\begin{renumerate}
\item
$V(G)= V(G_1) \sqcup V(G_2)$, 
\item
$C(i) =  C_1(i) \sqcup C_2(i)$.
for all $i \in [k]$.
\item
$E(G)= E(G_1) \sqcup E(G_2) \cup \{ (u,v) \in V(G): u \in C(i), v \in C(j)\}$,
which connects in the disjoint union all vertices in $C(i)$ with all vertices in $C(j)$.
\end{renumerate}
\fi 
$\bar{\eta}_{i,j}$ is a binary version of the operation $\eta_{i,j}$ used in the definition
of the clique-width of a graph, cf. \cite{pr:CourcelleMakowskyRoticsWG98}. 
\begin{proposition}
The operations $\sqcup_k$ and $\bar{\eta}_{i,j}$ are commutative and associative.
\end{proposition}

\subsection{Inductive definition of tree-width and clique-width}
As we do not need much of the theory of graphs of bounded tree-width and clique-width,
the following suffices for our purpose. The interested reader may consult \cite{ar:HlinenyOumSeeseGottlob08}.
In \cite{ar:MakowskyTARSKI} the following equivalent definitions of
the class of (labeled or colored) graphs of 
tree-width at most $k$ ($\TW(k)$),
path-width at most $k$ ($\PW(k)$),
clique-width  at most $k$ ($\CW(k)$), and
linear clique-width  at most $k$ ($\LCW(k)$)
were given:
\\
\ \\
\paragraph{\bf Tree-width}
\ifdense
\_

(i)
Every $k$-graph  of size at most $k+1$ is in $\TW(k)$ and $\PW(k)$.

(ii)
$\TW(k)$ is closed under disjoint union $\sqcup$ and  gluing $\sqcup_k$.

(iii)
$\PW(k)$ is closed under disjoint union $\sqcup$ and   small gluing $\sqcup_k$ where one operand is  
$k$-graph of size at most $k+1$.

(iv)
Let $\pi: [k] \rightarrow [k]$ be a partial relabeling function. 
If $(G, \ell) \in \TW(k)$ then also $(G, \ell') \in \TW(k)$ where $\ell'(i)= \ell(\pi(i))$.
The same holds for $\PW(k)$.

\else
\begin{renumerate}
\item
Every $k$-graph  of size at most $k+1$ is in $\TW(k)$ and $\PW(k)$.
\item
$\TW(k)$ is closed under disjoint union $\sqcup$ and  gluing $\sqcup_k$.
\item
$\PW(k)$ is closed under disjoint union $\sqcup$ and   small gluing $\sqcup_k$ where one operand is  
$k$-graph of size at most $k+1$.
\item
Let $\pi: [k] \rightarrow [k]$ be a partial relabeling function. 
If $(G, \ell) \in \TW(k)$ then also $(G, \ell') \in \TW(k)$ where $\ell'(i)= \ell(\pi(i))$.
The same holds for $\PW(k)$.
\end{renumerate}
If a graph $G$ is of tree-width at most $k$, there is a labeling $\ell$ such that
\\
$(G,\ell) \in \TW(k)$. Conversely, if $G \in \TW(k)$, then it is of tree-width at most $k+1$.
\fi 
\\
\ \\
\paragraph{\bf Clique-width}
\label{clique-width}
\ifdense
\_

(i)
Every single-vertex $k$-colored graph is in $\CW(k)$ and $\LCW(k)$.

(ii)
$\CW(k)$ is closed under disjoint union $\sqcup$ and  $(i,j)$-joins for $i,j \leq k$ and $ i \neq j$.

(iii)
$\LCW(k)$ is closed under disjoint union $\sqcup$ and small $(i,j)$-joins for $i,j \leq k$ and $ i \neq j$,
where one operand is a single-vertex $k$-colored graph.

(iv)
Let $\rho: 2^{[k]} \rightarrow 2^{[k]}$ be a recoloring function. 
If $(G, C) \in \CW(k)$ then also $(G, C') \in \CW(k)$ where $C'(I)= C(\rho(I))$.
The same holds for $\LCW(k)$.

\else
\begin{renumerate}
\item
Every single-vertex $k$-colored graph is in $\CW(k)$ and $\LCW(k)$.
\item
$\CW(k)$ is closed under disjoint union $\sqcup$ and  $(i,j)$-joins for $i,j \leq k$ and $ i \neq j$.
\item
$\LCW(k)$ is closed under disjoint union $\sqcup$ and small $(i,j)$-joins for $i,j \leq k$ and $ i \neq j$,
where one operand is a single-vertex $k$-colored graph.
\item
Let $\rho: 2^{[k]} \rightarrow 2^{[k]}$ be a recoloring function. 
If $(G, C) \in \CW(k)$ then also $(G, C') \in \CW(k)$ where $C'(I)= C(\rho(I))$.
The same holds for $\LCW(k)$.
\end{renumerate}
\fi 

\ifdense
\paragraph{}
\else
\fi 

A graph $G$ is of clique-width at most $2^k$ iff there is a coloring $C$ such that $(G,C) \in \CW(k)$.
The discrepancy between $2^k$ and $k$ comes from the fact that we allow overlapping colorings.
Note that in the original definition a unary operation $\eta_{i,j}$ is used 
instead of the binary $(i,j)$-join $\bar{\eta}_{i,j}$.  
However, the two are interdefinable with the help of disjoint union.
For a detailed discussion of various width parameters, cf. \cite{ar:HlinenyOumSeeseGottlob08}.

A {\em parse tree for $G$} is a witness for the inductive definition 
describing how $G$ was constructed.
Parse trees for $G \in \TW(k)$ and $G \in \PW(k)$ can be found in polynomial time, \cite{ar:BodlaenderKloks91}.
For $G \in \CW(k)$ the situation seems slightly worse.  It was shown in \cite{ar:Oum2005}:

\begin{proposition}[S. Oum]
\label{oum}
Let $G$ be a graph of clique-width at most $k$. Then we can find 
a parse tree for $G \in \CW(3k)$
in polynomial time.
\end{proposition}

\section{Graph parameters with values in a semiring and their Hankel matrices}
\label{se:alg-2}
\label{se:Hankel}
An $\cS$-valued graph parameter $f$ is a function $f: \cG \rightarrow \cS$ which is invariant under graph isomorphisms.
If we consider $f: \cG_k \rightarrow \cS$ or
$f: \cCG_k \rightarrow \cS$ then we require that $f$ is also invariant under labelings and colorings.

Let $X_i: i \in \N$ be an enumeration of all colored graphs in $\cCG_k$.
For a binary operation $\Box$ on labeled or colored graphs, and a graph parameter $f$, we define the
{\em Hankel matrix $\Ha(f, \Box)$} with 
$$\Ha(f, \Box)_{i,j} = f(X_i \Box X_j).$$
If the operation $\Box$ is $\sqcup_k$, the Hankel matrix $\Ha(f, \Box)$ is the connection matrix  $M(f,k)$ of \cite{bk:Lovasz-hom}.

Given a Hankel matrix $\Ha(f,\Box)$ we associate with it the semimodule $\MH(f, \Box)$ generated by its rows.
If there exist finitely many elements $g_1, \ldots, g_m \in \MH(f, \Box)$
which generated $\MH(f, \Box)$, we say that $\MH(f, \Box)$ is {\em finitely generated}.

\subsection{Notions of rank for matrices over semirings}
Semimodules over semirings are analogs of vector spaces over fields. 
However, 
in contrast to vector spaces, there are several ways of defining the notion of independence 
for semimodules. For our purposes we adopt the definition 3.4 used in \cite[Section 3]{bk:Guterman2009} and in
\cite{ar:Cuninghame-GreenButkovic2004}, but see also
\cite{ar:DevelinSantosSturmfels2005,ar:AkianGaubertGuterman2009}.
A set of elements $P$ from a semimodule $U$ over a semiring $\cS$ is 
\emph{linearly independent} if there is no element in $P$ that can be expressed 
as a linear combination of other elements in $P$.

Using this notion of linear independence, we define the notions of basis and 
dimension as in \cite{bk:Guterman2009,ar:Cuninghame-GreenButkovic2004}: 
a \emph{basis} of a semimodule $U$ over 
a semiring $\cS$ is a set $P$ of linearly independent elements from $U$ which generate it, 
and the \emph{dimension} of a semimodule $U$ is the cardinality of its smallest basis.

Given a Hankel matrix $\Ha(f,\Box)$ with its associated semimodule $\MH(f, \Box)$, 
we define the \emph{row-rank $r(\Ha(f,\Box))$} of the matrix as the dimension of $\MH(f, \Box)$. 
In addition, we say that $\Ha(f,\Box)$ has \emph{maximal row-rank $mr(\Ha(f,\Box)) = k$} 
if $\Ha(f,\Box)$ has $k$ linearly independent rows and any $k+1$ rows are linearly dependent. 
These definitions are the definitions used in \cite{bk:Guterman2009}, 
applied to infinite matrices.

As stated in \cite{bk:Guterman2009,ar:Cuninghame-GreenButkovic2004},
in the case of tropical semirings, we have $$r(\Ha(f,\Box)) = rm(\Ha(f,\Box)).$$

\begin{lemma}
\label{le:fg}
If a Hankel matrix $\Ha(f,\Box)$ over a tropical semiring has row-rank $r(\Ha(f,\Box)) = m$, 
then there are $m$ rows in $\Ha(f,\Box)$ which form a basis of $\MH(f, \Box)$.
\end{lemma}

\begin{remark}
If the matrix $\Ha(f,\Box)$ is over a general semiring $\cS$, a smallest basis of 
$\MH(f, \Box)$ does not necessarily reside in $\Ha(f,\Box)$.
\end{remark}

\begin{proof}
$r(\Ha(f,\Box)) = m$, so by definition the dimension of $\MH(f, \Box)$ is $m$.
Suppose the set $\cB = \{g_1, \ldots, g_m\}$ is a smallest basis for 
$\MH(f, \Box)$. Each $g_p$ is in $\MH(f, \Box)$, 
therefore there is a finite linear combination of rows from $\Ha(f,\Box)$ such that 
\\
$g_p = \bigoplus_{i_p=1}^{\ell_p}{\alpha_{i_p} r_{i_p}}$. 
Consider the set of all the rows that appear in any of these linear combinations: 
$\cR = \bigcup_{p=1}^m{\lb(\bigcup_{i_p=1}^{\ell_p}{\alpha_{i_p} r_{i_p}}\rb)}$.
Since $\Ha(f,\Box)$ is over a tropical semiring, it holds that $mr(\Ha(f,\Box)) = r(\Ha(f,\Box)) = m$.
Therefore, any set of $m+1$ rows from $\Ha(f,\Box)$ is linearly dependent. 
Consider the 
result
of the following process: 

\ifdense

(i)
Set $i=|\cR|$, and $B_i = \cR$, note that $B_i$ is of size $i$ and generates $\cB$. Repeat until $i=m$:

(ii)
Let $r' \in B_i$ be a row that can be expressed using other rows in $B_i$. 
Such an element must exist, as $|B_i|>m$. 
Set $B' = B_i - {r'}$, set $i = i-1$ and $B_i = B'$. 

Note that $B_i$ is still of size (now smaller) $i$ and it still generates $\cB$.

\else
\begin{itemize}
\item
Set $i=|\cR|$, and $B_i = \cR$, note that $B_i$ is of size $i$ and generates $\cB$. Repeat until $i=m$:
\item
Let $r' \in B_i$ be a row that can be expressed using other rows in $B_i$. 
Such an element must exist, as $|B_i|>m$. 
Set $B' = B_i - {r'}$, set $i = i-1$ and $B_i = B'$. 
Note that $B_i$ is still of size (now smaller) $i$ and it still generates $\cB$.
\end{itemize}
\fi 
When $i=m$ is reached, we have $B_m$ of size $m$ which generates $\cB$. 
This set must be independent: 
if it were not, we could perform more iterations of the above process and 
obtain a linearly independent set of size $< m$ which generates $\cB$. 
But the existence of such a set contradicts $\cB$ being a smallest basis, 
Therefore, $B$ is linearly independent and generates $\cB$. 
Since $\cB$ generate $\MH(f, \Box)$, so does $B$, making $B$ a basis for $\MH(f, \Box)$
which resides in $\Ha(f,\Box)$.
\end{proof}

After establishing the fact that there lies a basis $B$ of $\MH(f, \Box)$ in $\Ha(f,\Box)$, 
we can find it in finite time, due to 
\cite[Theorems 2.4 and 2.5]{ar:Cuninghame-GreenButkovic2004}.
\section{Graph parameters with join matrices of finite (row-)rank}
\label{se:many}

\subsection{Graph parameters definable in Monadic Second Order Logic}
It follows from \cite{ar:GodlinKotekMakowsky08,pr:KotekMakowsky2012,ar:KotekMakowsky2013}
that for graph parameters definable in Monadic Second Order Logic (MSOL)
or MSOL with modular counting quantifiers (CMSOL), the connection matrices and join matrices
all have finite rank over fields, and finite row-rank over tropical semirings.

Let $H= (V(H), E(H))$ be a weighted graph with weight functions on vertices and edges
$\alpha: V(H) \rightarrow \R$ and $\beta: E(H) \rightarrow \R$.
The {\em tropical partition function $Z_{H,\alpha,\beta}$} on graphs $G$ is defined by
\begin{gather}
Z_{H,\alpha,\beta}(G) = 
\bigoplus_{h:G \rightarrow H} \left( \bigotimes_{v \in V(G)} \alpha(h(v)) 
\bigotimes_{(u,v) \in E(G)} \beta(h(u),h(v)) \right)
\notag 
\end{gather}
In the tropical ring this can be written as:
\begin{gather}
Z_{H,\alpha,\beta}(G) = 
\max_{h:G \rightarrow H} \left( \sum_{v \in V(G)} \alpha(h(v)) 
\sum_{(u,v) \in E(G)} \beta(h(u),h(v)) \right)
\notag
\end{gather}
where $h$ ranges over all homomorphisms $h: G \rightarrow H$.

It is easy to verify that $Z_{H,\alpha,\beta}$ is $\MSOL$-definable. Hence we have:

\begin{proposition}
$\Ha(Z_{H,\alpha,\beta}, \bar{\eta}_{i,j})$ has finite row-rank.
\end{proposition}
The independence number $\alpha(G)$, 
which is the cardinality of the largest independent set, is a special case of a tropical partition function.

There are many graph parameters which have infinite connection rank, but finite
row-rank if interpreted over tropical semirings.
Examples for this phenomenon are the clique number $\omega(G)$ and the independence number $\alpha(G)$.
Many other examples may be found in \cite{ar:ArnborgLagergrenSeese91}.

\subsection{Uncountably many graph parameters with finite (row-)rank}
Here we show that both over fields and tropical semirings, most of the graph parameters with
finite (row-)rank of connection or join matrices are not definable in the above mentioned logics.

We first need an observation.
A graph is {\em $k$-connected}, if there is no set of $k$ vertices, such that
their removal results in a graph which is not connected.
Obviously we have:
\begin{lemma}
\label{le:k-connected}
Let $G_1$ and $G_2$ be two $k$-graphs and $G= G_1 \sqcup_k G_2$.
Then $G$ is not $k+1$-connected.
\end{lemma}

For a subset $A \subseteq \N$ we define graph parameters
$$
f_A(G) =
\begin{cases}
|V(G)| & G \mbox{ is  } k_0+1 \mbox{-connected and } |V(G)| \in A \\
0 & \mbox{ else}
\end{cases}
$$

\begin{lemma}
Let $\sr$ be a commutative semiring which contains $\N$.
Let $k_0 \in \N$ and $A \subseteq \N$ with $1 \in A$.
Then for every $k \leq k_0$ the semimodule of the rows of
$\Ha(f_A, \sqcup_k)$  is generated by the  two rows
$$
(1, 0, \ldots ) \mbox{    and    }
( \ldots, f_A(\emptyset, G), \ldots )
$$
If $\sr$ is a field,
$\Ha(f_A, \sqcup_k)$  has rank at most $2$.
\end{lemma}
\begin{proof}
By Lemma \ref{le:k-connected}, if the graph $G_1 \sqcup_k G_2$ is $k_0+1$-connected, then
either $G_1$ is $k_0+1$-connected and $G_2$ is the empty graph, or vice versa.
So the non-zero entries in $\Ha(f_A, \sqcup_k)$ are in the first row and the first column.
As $1 \in A$, we have a row $(1, 0, \ldots )$ which generates all the rows but the first one.
\end{proof}

\begin{theorem}
\label{th:many-s}
Let $k_0 \in \N$ and $\sr$ a field.
There are continuum many graph parameters $f$ with values in $\sr$ with $r(f, \sqcup_k) \leq 2$ for each $k \leq k_0$.
\\
The same holds for tropical semirings and row-rank.
\end{theorem}
\begin{proof}
There are continuum many subsets $A \subseteq \N$ and for two different sets $A,B \subseteq \N$ 
the parameters $f_A$ and $f_B$ are different.
\end{proof}

\ifskip 
\else
Let $G_1, G_2$ be two $1$-graphs (with one distinguished element $v_1, v_2$ respectively).
We define the operation $subst(G_1, G_2) =G$ by
\begin{gather}
V(G) =  V(G_1) \sqcup V(G_2) - \{ v_1\} \notag \\
E(G) = (E(G_1) \sqcup E(G_2) - \{ (u, v_1): u \in V(G_1) \})  \notag \\
\cup \{ (u,v): u \in V(G_1)-\{v_1\}, (v_1, v) \in E(G_2) \} 
\notag
\end{gather}

The operation is used in \cite{ar:Specker88}, where the following is shown:
\begin{theorem}[Specker]
\label{th:specker}
There are continuum many graph properties $P_A$ such that
$\Ha(P_A, subst)$ has rank at most $5$.
\end{theorem}
The operation $subst(G_1, G_2)$ can be used to reconstruct a graph
from its modular decomposition.
\fi 

Let $(G_1, C_1), (G_2, C_2)$ be two $2$-colored graphs. 

\begin{lemma}
\label{le:eta}
Let $(G,C) =\bar{\eta}_{1,2}((G_1,C_1), (G_2,C_2))$ and let $C_1(2) = C_2(1)= \emptyset$.
\ifdense

(i) $G$ is the disjoint union of $G_1$ and $G_2$ iff $C_1(1)$ or $C_2(2)$ are empty.

(ii) If both $C_1(1)$ and $C_2(2)$ are not empty, there is a vertex in $C_1(1)$ which has a higher degree
in $G$ than it had in $G_1$.
\else
\begin{renumerate}
\item
$G$ is the disjoint union of $G_1$ and $G_2$ iff $C_1(1)$ or $C_2(2)$ are empty.
\item
If both $C_1(1)$ and $C_2(2)$ are not empty, there is a vertex in $C_1(1)$ which has a higher degree
in $G$ than it had in $G_1$.
\end{renumerate}
\fi 
\end{lemma}

Let $r \in \N$ and $A \subseteq \N$.
We define graph parameters with values in $\N$:
$$
g_A^r(G) =
\begin{cases}
|V(G)| & G \mbox{ is  } r \mbox{-regular and connected and } |V(G)| \in A \\
0 & \mbox{ else}
\end{cases}
$$

\begin{theorem}
\label{th:many-c}
Let 
$\sr$ be a field of characteristic $0$.
There are continuum many graph parameters $g_A^r$ with values in $\sr$ such that $r(f, \bar{\eta}_{1,2}) \leq 2$.
\\
Similarly for commutative semirings.
\end{theorem}
\begin{proof}
Use Lemma \ref{le:eta}.
\end{proof}
\section{Graph algebras}
\label{se:algebras}
This section presents our adaptation of the formalism of graph algebras, to tropical semirings with an inner product derived 
from the join matrices of tropical graph parameters.

\subsection{Quantum graphs}

A formal linear combination of a finite number of $k$-colored graphs $F_i$ with coefficients 
from $\cT_{max}$ ($\cT_{min}$) is called a \emph{quantum graph}. 
The set of $k$-colored\footnote{
In \cite{bk:Lovasz-hom} this notations is used only for $k$-graphs and real coefficients.
As $k$-graphs are a special case of $k$-colored graphs  our notations also includes his.
} 
quantum graphs is denoted $\cQ_k$. 

Let $X,Y$ be quantum graphs: $X=\bigoplus_{i=1}^{m}{a_i F_i}$, 
and $Y=\bigoplus_{i=1}^{n}{b_i F_i}$. Note that some of the coefficients may be $-\infty$ ($\infty$).

$\cQ_k$ is a semimodule with the operations:
\begin{itemize}
\item 
$x \oplus y = 
\left( \bigoplus_{i=1}^{m}{a_i F_i} 
\right) 
\oplus \left( \bigoplus_{i=1}^{n}{b_i F_i} \right) = 
\bigoplus_{i=1}^{\max\{m,n\}}{(a_i \oplus b_i) F_i}$, 
and 
\item
$\alpha \otimes x = \bigoplus_{i=1}^{n}{(\alpha \otimes a_i)F_i}$
\end{itemize}

We extend any binary operation $\Box$ to quantum graphs by
$$
\Box(X,Y) =
\bigoplus_{i,j=1}^{m,n} (a_i \otimes b_j) \Box(F_i, F_j) 
$$

We extend any graph parameter $f$ to quantum graphs linearly
$$
f(X)=
\bigoplus_{i=1}^{m}{a_i f(F_i)}
$$

From now on we assume that $\Box$ is a {\em commutative} graph operation.
Given a Hankel matrix $\Ha(f,\Box)$, we turn $\cQ_k$ into a commutative algebra by defining an inner product on $X,Y$:
\begin{align*}
\langle X,Y \rangle_{f,\Box} &
= f(\Box(X,Y)) = 
\bigoplus_{i,j=1}^{m,n}{\left( (a_i \otimes b_j) \otimes f(\Box(F_i,F_j)) \right)}
\end{align*}

\subsection{Equivalence relations over $\cQ_k$}

Given a Hankel matrix $\Ha(f,\Box)$, we define an equivalence relation in the following way:
$$
Ker_f^{\Box} = \{(X,Y) \in \cQ_k \times \cQ_k| \forall Z \in \cQ_k : f(\Box(X,Z)) = f(\Box(Y,Z)) \}
$$
Note that this definition is reminiscent to the equivalence relation used in the Myhill-Nerode Theorem
characterizing regular languages, cf. \cite{bk:HU}.

We denote the set of equivalence classes of this relation by $\cQ_k / Ker_f^{\Box}$.
$\cQ_k / Ker_f^{\Box}$ is a semimodule with the operations:
$$[X]_f^{\Box} \oplus [Y]_f^{\Box} = [X \oplus Y]_f^{\Box}$$
and
$$\alpha [X]_f^{\Box} = [\alpha X]_f^{\Box}$$

We turn $\cQ_k / Ker_f^{\Box}$ into a quotient algebra by
extending the binary operation $\Box$ to these equivalence classes. We define
$$
\Box([X]_f^{\Box},[Y]_f^{\Box}) = [\Box(X,Y)]_f^{\Box}
$$

It can be easily verified that the following properties hold for $X' \in [X]_f^{\Box}$ and $Y' \in [Y]_f^{\Box}$:
\begin{proposition}
\label{prop:comm-assoc}
Let $\Box$ be a commutative and associative operation on graphs.
\begin{renumerate}
\item 
$X' \oplus Y' \in [X \oplus Y]_f^{\Box} = [X]_f^{\Box} \oplus [Y]_f^{\Box}$
\item 
$\alpha X' \in [\alpha X]_f^{\Box} = \alpha[X]_f^{\Box}$
\item 
$\Box(X',Y') \in [\Box(X,Y)]_f^{\Box} $
\end{renumerate}
\end{proposition}

\subsection{Finiteness condition on Hankel matrices}
\ifskip
\else
Now we can complete the groundwork needed for proving Theorem \ref{th:cw-poly}. 

We denote by $\Ha_k(f, \eta)$ the Hankel matrix of $f$ for $k$-colored graphs with the join operation, 
and by $\MH_k(f, \eta)$ we denote the semimodule associated with it.

\begin{lemma}
If $\MH_k(f, \eta)$ is finitely generated for an isolate-indifferent $f$, then so are $\MH_{\ell}(f, \Box)$ for $\ell \leq k$.
\end{lemma}

\begin{proof}
Consider the rows $\cI$ of $\Ha_k(f, \eta)$ that correspond to graphs with some isolated vertex which 
is the only vertex colored $k$. These rows $\cI$ are identical to the rows in $\Ha_{k-1}(f,\eta)$, 
since $f$ is isolate indifferent. 
$\MH_k(f,\eta)$ is finitely generated, and in particular $\cI$ are finitely generated by the same generating set. 
Therefore, so are the elements in $\MH_{k-1}(f,\eta)$.
\end{proof}
\fi 

Given the Hankel matrix $\Ha(f, \Box)$ we denote by
$\MH(f, \Box)$ the semimodule generated by the rows of $\Ha(f, \Box)$.

\begin{lemma}
\label{le:generated}
Assume the semimodule $\MH(f, \Box)$ is generated by the rows $Gen = \{r_1, \ldots, r_m\}$ 
of $\Ha(f, \Box)$, where each row corresponds to a graph $G_q$. 
Then $\cQ_k/Ker_f^{\Box}$ is generated by $\cB_k = \{[G_1]_f^{\Box}, \ldots, [G_m]_f^{\Box}\}$.
\end{lemma}

\begin{proof}
Let $X \in \cQ_k$, where $X = \bigoplus_{i=1}^{n}{a_i F_i}$. 
Each $F_i$ a linear combination of the generators $G_1, \ldots, G_m$, $F_i = \bigoplus_{j=1}^{m}{\alpha_{i,j} G_j}$.
By Proposition \ref{prop:comm-assoc}(i)-(ii) we have 
$$
X \in \bigoplus_{i=1}^{n}{a_i [F_i]_f^{\Box}}
=
\bigoplus_{i=1}^{n}{a_i \bigoplus_{j=1}^{m}{\alpha_{i,j} [G_j]_f^{\Box}}}
$$
\end{proof}

\section{Graphs of clique-width at most $k$}
\label{se:CW}

Let $k \in \N$ be fixed.
From now on $\Box = \bar{\eta}_{1,2}$ on $k$-colored graphs, and we write simply $\eta^k$
instead of $\bar{\eta}_{1,2}$.
We omit $k$ when it is clear from the context.
The Hankel matrix $\Ha(f, \eta)$ is called the join matrix.

We note that because the rows and columns correspond to  all the graphs with all the possible $k$-colorings, all the join matrices $\Ha(f, \bar{\eta}_{i,j})$ are submatrices of $\Ha(f, \eta)$,
after a suitable recoloring.

\subsection{Representing $G$ in the graph algebra}
Given a graph $G$ of clique-width at most $k$, together with its 
parse tree
of the inductive definition
from  Section \ref{clique-width}, we want to find 
$[X_G]_f^{\eta} \in \cQ_k/Ker_f^{\eta}$ s.t. $f(X_G) = f(G)$.
Furthermore, $[X_G]_f^{\eta}$ will be a linear combination of generators of $\cQ_p/Ker_f^{\eta}$
and 
will be computable in polynomial time.

The same result for tree-width follows from the result on clique-width, but it can 
also
be 
directly
obtained 
using the inductive definition of tree-width from Section \ref{clique-width}.

\begin{lemma}
\label{le:rep}
Let $G$ of clique-width at most $k$ be given together with its parse tree $T$, and let 
$\cB = \{[F_{1}]_f^{\eta}, \ldots, [F_{m}]_f^{\eta} \}$ be a basis of $\cQ_k/Ker_f$. 
Then there exists 
$[X_G]_f^{\eta} \in \cQ_k/Ker_f^{\eta}$ 
s.t. 
$f(X_G) = f(G)$, and $[X_G]_f^{\eta}$ 
can be represented as a linear combination of 
$\{[F_{1}]_f^{\eta}, \ldots, [F_{m}]_f^{\eta} \}$.
\end{lemma}

\begin{proof}

Let $S_1, \ldots, S_\ell \in \cQ_k$ be the single-vertex $k$-colored graphs (we later refer to them as small graphs),
and let 
$[S_i]_f^{\eta} = \bigoplus_j{s_{ij}[F_{j}]_f^{\eta}}$ be their representations in the basis $\cB$.
\\
Let $\eta([F_{i}]_f^{\eta},[F_{j}]_f^{\eta})$ be the representations of the 
results of the $\eta$ operation on 
elements from the basis $\cB$.
\\
Let $G$ be a graph of clique-width at most $k$, and let $T$ be its parse tree.
We proceed by induction on $T$.
\\
If $G = S_i$ then set $X_G  = S_i$. The graph $S_i$ is a single-vertex graph, and we have a representation for 
$X_G$.
\\
Assume that for $G_1, G_2$, there exist 
$[X_{G_1}]_f^{\eta},[X_{G_2}]_f^{\eta}$ 
and that there are representations 
$[X_{G_1}]_f^{\eta} = \bigoplus_{i=1}^{n}{a_i [F_i]_f^{\eta}}$,  
$[X_{G_2}]_f^{\eta} = \bigoplus_{i=1}^{m}{b_i [F_i]_f^{\eta}}$ for them.
\\
If $G = \eta(G_1,G_2)$, then by Proposition \ref{prop:comm-assoc}(iii) we have 
$$\eta(G_1,G_2) \in [\eta(G_1,G_2)]_f^{\eta} = \eta([X_{G_1}]_f^{\eta},[X_{G_2}]_f^{\eta}).$$ 
We have  representations for the operations $\eta([F_{i}]_f^{\eta},[F_{j}]_f^{\eta})$ on the basis elements,
so we replace the expressions $\eta([F_{i}]_f^{\eta},[F_{j}]_f^{\eta})$ in 
$\eta([X_{G_1}]_f^{\eta},[X_{G_2}]_f^{\eta})$ by these representations and obtain a representation
of 
$$X_G = \eta(G_1,G_2) \in [\eta(G_1,G_2)]_f^{\eta}.$$
\\
If $G = \rho_{i,j}(G_1)$, we replace the basis elements $[F_{i}]_f^{\eta}$
in the representation of $X_{G_1}$ by the representations of $[\rho_{i,j}(F_{i})]_f^{\eta}$
and obtain a representation of $X_G$.
\end{proof}

\subsection{Computing $f(G)$}
\begin{theorem}
\label{th:cw-poly}
Let $f$ be a tropical graph parameter. 
Let the row-rank $r(\Ha(f,\eta))$ be finite, and let $G$ be a graph of clique width at most $k$. 
Then $f(G)$ can be computed in polynomial time. 
\end{theorem}
\begin{proof}
We first use Proposition \ref{oum} to find a parse tree for $G \in \CW(3k)$.
Next, we use dynamic programming to build a representation of 
the given graph $G$ in the basis $\cB$ in order to obtain $f(G)$. 
The algorithm requires a finite amount of preprocessing:
\\
Find basis elements $\cB$. By 
Lemma \ref{le:fg} and Theorems  2.4 and 2.5 in \cite{ar:Cuninghame-GreenButkovic2004}
this can be done in finite time.
\\
Compute representations of all the small graphs by basis elements
\\
Compute representations of the product $\eta$ for all basis elements and all small graphs
\\
Compute the value of $f$ on all the small graphs and basis elements
\\
The algorithm works with the provided parse tree $T$ from the bottom up, following the inductive definition 
given in the proof of Lemma \ref{le:rep}.
When the top of the tree is reached, we have a representation
of 
$[X_G]_f^{\eta}$
using only basis elements $[F_i]_f^{\eta}$.
We then use the precomputed values of $f$ in order to compute the value of $f(X_G) = f(G)$.
\end{proof}

\subsection{Commutative semirings}

In the case of $\cS$ being an arbitrary commutative semiring we use following finiteness condition
first introduced in \cite{phd:Jacob}:

A Hankel matrix $\Ha(f,\Box)$ of an $\cS$-valued graph parameter $f$ is {\em J-finite}
if $MH(f, \Box)$ is finitely generated. This does not necessarily imply that
$\Ha(f, \Box)$ has a finite row-rank. However, in automata theory it suffices to prove the following:
Let $f$ be a $\cS$-valued function on words in $\Sigma^*$ (for a finite alphabet $\Sigma$).
Then $f$ is recognizable by a multiplicity automaton iff $\Ha(f, \circ)$ is J-finite, \cite{bk:BerstelReutenauer}.
Using virtually the same proof we can show:

\begin{theorem}
\label{th:lcw-poly}
Let $\cS$ be an arbitrary commutative semiring.
Let $f$ be an $\cS$-valued graph parameter and $k \in \N$ be fixed.
\begin{renumerate}
\item
If $\Ha(f, \sqcup_i)$ is J-finite for all $i \leq k$, then $f$ can be computed in polynomial time
on graphs of path-width at most $k$.
\item
If $\Ha(f, \eta^k)$ is J-finite, then $f$ can be computed in polynomial time
on graphs of linear clique-width at most $k$.
\end{renumerate}
\end{theorem}
\section{Conclusions}
\label{se:conclu}

L. Lov\'asz showed
a ``logic-free'' version of Courcelle's famous theorem, \\
cf.  \cite[Chapter 6.5]{bk:DowneyFellows99} and \cite[Chapter 11.4-5]{bk:FlumGrohe2006}.

\begin{theorem}[Theorem 6.48 of \cite{bk:Lovasz-hom}]
\label{th:tw}
Let $f$ be a  real-valued graph parameter and $k \geq 0$. 
If $r(f, \sqcup_k)$ is finite, then $f$ can be computed in polynomial time for graphs
of tree-width at most $k$.
\end{theorem}
The proof in \cite{bk:Lovasz-hom} is rather sketchy in its part relating to tree-decompositions.
In particular, the role of relabelings, admittedly not very critical, is not spelled out at all.

In this paper we extended Theorem \ref{th:tw} to Theorems 
\ref{th:cw-poly}
and
\ref{th:lcw-poly}
in two ways.
\ifdense

(i)
We showed how to prove the theorem for bounded clique-width instead of bounded tree-width.

(ii)
We showed how to prove the theorem for 
tropical graph parameters, and more generally
for graph parameters in an arbitrary commutative semirings.
\else
\begin{renumerate}
\item 
We showed how to prove the theorem for bounded clique-width instead of bounded tree-width.
\item 
We showed how to prove the theorem for 
tropical graph parameters, and more generally
for graph parameters in an arbitrary commutative semirings.
\end{renumerate}
\fi 

In order to do this we introduced Hankel matrices for binary graph operations,
in particular for a binary version of the basic operations used in the definition of clique-width.

The main differences between our proofs and the proof in \cite{bk:Lovasz-hom} are:
\ifdense

(i)
the definition of the join matrix, 

(ii)
the choice of the finiteness condition of the join matrices, and 

(iii)
the choice of the definition of the equivalence relation used for the quotient algebra.
\else
\begin{renumerate}
\item %
the definition of the join matrix, 
\item %
the choice of the finiteness condition of the join matrices, and 
\item %
the choice of the definition of the equivalence relation used for the quotient algebra.
\end{renumerate}
\fi 

We also had to spell out the role of parse trees for clique-width 
in the dynamic programming part of the polynomial time algorithm.

Our approach also works for 
\ifdense

$\bullet$\ 
{\em other notions of width for graphs}, such as rank-width and modular width,
and other inductively defined graph classes, cf. \cite{ar:MakowskyTARSKI,ar:HlinenyOumSeeseGottlob08}.

$\bullet $\ 
{\em other notions of connection matrices}, cf. \cite{ar:Schrijver-spin,ar:Schrijver-vertex}.

\else
\begin{itemize}
\item
{\em other notions of width for graphs}, such as rank-width and modular width,
and other inductively defined graph classes, cf. \cite{ar:MakowskyTARSKI,ar:HlinenyOumSeeseGottlob08}.
\item
{\em other notions of connection matrices}, cf. \cite{ar:Schrijver-spin,ar:Schrijver-vertex}.
\end{itemize}
\fi 
In the full paper we shall discuss these extensions in detail.

Tropical graph parameters occur naturally in optimization theory.
Graph parameters with values in polynomial rings are called graph polynomials,
and are widely studied in diverse fields as statistical mechanics, computational biology and mathematics of finance.
It remains open to identify the most suitable finiteness condition on Hankel matrices in the case where
the graph parameter has its values in a arbitrary ring or semiring.
\label{sec:biblio}

\end{document}